\documentclass[11pt]{amsart}
\usepackage{amssymb}
\usepackage{amsthm}
\usepackage{amsmath}
\usepackage{hyperref}

\usepackage{url}
\newtheorem{theorem}{Theorem}
\newtheorem{lemma}[theorem]{Lemma}

\newtheorem{definition}[theorem]{Definition}
\newtheorem{proposition}[theorem]{Proposition}
\newtheorem{claim}{Claim}

\newcommand{\z}{\mathbb Z}

\newcommand{\X}{\mathcal{X}}
\newcommand{\Y}{\mathcal{Y}}

\begin{document}

\title{On the Removal Lemma for  Linear Systems over Abelian
Groups}

 \author{Daniel Kr\'al'}
 \address{Computer Science Institute,
         Faculty of Mathematics and Physics, Charles University.}
\email{ kral@kam.mff.cuni.cz}
\thanks{The work of this author, leading to this invention, has received founding from European Research Council under the European Union's Seventh Framework Programme (FP7/2007-2013)/ERC Grant Agreement no. 259385.}%

 \author{Oriol Serra}
 \address{Departament de Matem\`atica Aplicada IV,
        Universitat Polit\`ecnica de Catalunya.}
	 \email{oserra@ma4.upc.edu}
\thanks{Supported by the Catalan Research Council   under project 2008SGR0258 and the Spanish Research Council under project MTM2011-28800-C02-01.}%
 \author{Llu\'{i}s Vena}
 \address{
 Department of Mathematics, University of Toronto.}
 \email{lluis.vena@utoronto.ca}
\thanks{Supported by  a University of Toronto Graduate
Fellowship.}

\maketitle

\begin{abstract}
In this paper we present an extension of the removal lemma  to
integer linear systems over abelian groups. We prove that,  if   the
$k$--determinantal of an integer $(k\times m)$ matrix $A$ is coprime
with the order $n$ of a group $G$ and the number of solutions of the
system $Ax=b$ with $x_1\in X_1,\ldots, x_m\in X_m$ is
$o(n^{m-k})$, then we can eliminate $o(n)$ elements in each
set to remove all these solutions.
\end{abstract}



\section{Introduction}

In 2005 Green~\cite{Green05} introduced the so-called Removal Lemma
for Groups. Roughly speaking, this result states that if a linear equation 
 $$ 
 x_1+x_2+\cdots +x_m=0
 $$ 
  has not many
solutions with variables taking values from given subsets $X_1,\ldots\linebreak[1] ,X_m$
of a finite Abelian group $G$, then one can delete all these
solutions by removing few elements in each subset.
This result is inspired by the removal lemma for triangles in graphs  (see
\cite{SzemRusz76}).

The Removal Lemma for Groups has been extended
 to one equation
with elements in non-necessarily Abelian groups by the authors
\cite{KralSerraVena09} and, by confirming a conjecture of
Green~\cite{Green05}, to linear systems over Finite Fields
independently by Shapira  \cite{Shap09-1} and the authors
\cite{KralSerraVena09-2}. Shapira \cite{Shap09-1}   asked for an extension of the result
to Abelian groups. This work attempts to answer this
question.

Recall that the $k$-th determinantal 
divisor $d_k(A)$ of an integer matrix $A$ is
the greatest common divisor of the determinants of all the $k\times k$ 
submatrices of
$A$ obtained by selecting $k$ not necessarily consecutive rows and columns. This notion appears in the description of the Smith Normal Form of integer matrices; see e.g. Newman \cite{new72}. For simplicity, we use the shorter term 
{\it $k$-th determinantal} instead of $k$-th determinantal divisor. 
Our  main result is the following:

\begin{theorem}\label{t.mainb} Let $m, k$ be positive integers with $m\ge k$. For any $\epsilon>0$ there exists  $\delta > 0$ which depends on $\epsilon$ and $m$ such that the following holds.

Let $A$ be a $k\times m$ integer matrix $A$ and let $G$ be an  Abelian group $G$ of order $n$ coprime with $d_k(A)$. For
every family of subsets $X_1,\ldots ,X_m$ of $G$ and for every
vector  $b\in G^k$, if the linear system $Ax=b$ has at most $\delta
n^{m-k}$ solutions with $x_1\in X_1,\ldots ,x_m\in X_m$ then there
are sets $X'_1\subset X_1,\ldots ,X'_m\subset X_m$ with $|X'_i|\leq
\epsilon n$, for all $i$, such that there is no solution of the
system with $x_1\in X_1\setminus X'_1,\ldots, x_m\in X_m\setminus
X'_m$.
\end{theorem}

In the little `o' notation, Theorem~\ref{t.mainb} states that, if an
integer linear system over an Abelian group of order $n$, with $\gcd(n,d_k(A))=1$, 
has $o(n^{m-k})$ solutions, then we can
destroy all the solutions by removing $o(n)$ elements in each set.

Theorem \ref{t.mainb} is analogous to the statement for linear systems in finite fields proved in \cite{KralSerraVena09-2,Shap09-1} except that the condition on the $k$-determinantal is substituted there by the matrix $A$ having full rank.
The full rank condition can be easily removed from the hypothesis of \cite[Theorem~1]{KralSerraVena09-2} by a straightforward argument. The analogous condition that $d_k(A)=0$ in Theorem \ref{t.mainb} can  be similarly removed. However, the condition that $d_k(A)$ be coprime with $n$, being relatively natural,  cannot be easily removed. We refer the reader to the last section for a discussion on this issue.

A general framework for the study of this type of results  is
discussed by Szegedy~\cite{SzegSRL09}. The author proves a Symmetry-preserving
removal lemma  and applies it to give a diagonal version of the
Szemer\'edi Theorem on arithmetic progressions in Abelian groups.
Our work follows the direction of our original argument for the
nonabelian case presented in \cite{KralSerraVena09}, and it provides
a general answer for   linear systems $Ax=b$ which includes the
case of arithmetic progressions \cite[Theorem 3]{SzegSRL09}.

The proof of Theorem \ref{t.mainb} uses the removal lemma for
colored hypergraphs. An $r$-colored $k$-uniform hypergraph is a pair $(V,E)$ formed by a
set $V$ of vertices and a subset $E\subset {V \choose k}$ of edges,
which are $k$--subsets of vertices, and a map $c:E\rightarrow [1,r]$
which assigns `colors' to the edges. 

Given two colored $k$--uniform
hypergraphs $H$ and $K$, we say that $K$ contains a copy of $H$ if
there is an injective homomorphism from $H$ to $K$, that is, a map $f:V(H)\to V(K)$ whose natural
extension to edges preserves edges and colors. We also say that $K$
contains two disjoint copies of $H$ if there are two injective
homomorphisms $f, f'$ from $H$ to $K$ such that $f(E(H))\cap
f'(E(H))=\emptyset.$ The hypergraph $K$ is $H$--free if it contains
no copy of $H$. 

Extensions of the removal lemma to
hypergraphs have been obtained by several authors, see  Austin and Tao
\cite{AustinTAo09}, Elek and Szegedy \cite{ElekSzeg08}, Gowers
\cite{Gowers07}, Ishigami \cite{Ishigami} or  Nagle, R\"odl,
Schacht and Skokan \cite{Rodl06,rodsko06}. We shall use the following  version of the
hypergraph Removal Lemma, which follows, for instance, from
\cite[Theorem~1.5]{austao10}.

\begin{theorem}
\label{th:rmhc}
 For every positive integers $r$,   $m\ge k\ge 2$ and every $\epsilon >0$ there is   a $\delta>0$
 depending on $r$, $m$, $k$ and $\epsilon$  such that
the following holds.

Let $H$ and $K$  be  $r$-colored $k$-uniform hypergraphs with
    $m=|V(H)|$ and   $M=|V(K)|$ vertices respectively.
If the number of copies of $H$ in $K$ (preserving the colors of the
edges) is at most $\delta M^m$, then there is a set $E'\subseteq
E(K)$ of size at most $\epsilon M^{k}$ such that the hypergraph
$K'$ with edge set $E(K)\setminus E'$ is $H$--free.
\end{theorem}

The plan of the paper is as follows. In Section \ref{s.prf} we prove Theorem \ref{t.mainb} for a special class of matrices, which we call {\it standard $n$--circular}. The proof consists of associating to the system a pair of edge--colored hypergraphs in order to  transfer the statement to a setting in which the removal lemma for hypergraphs can be applied. Section \ref{s.prflem} introduces the notion of {\it restricted linear system} and provides the means to   transfer the result to general linear systems. The main result of the section is synthetized in Proposition \ref{prop:ext}. The proof of the main result is completed in Section \ref{sec:proof}. A closing section is devoted to discuss the condition on the $k$--determinantal in Theorem \ref{t.mainb}, which is not present in the analogous result for finite fields \cite[Theorem~1]{KralSerraVena09-2}.

\section{Standard $n$--Circular  Matrices} \label{s.prf}

In this section we prove Theorem~\ref{t.mainb} in the particular case of
homogeneous linear systems with what we call standard  $n$--circular 
matrices for Abelian groups of order $n$. We
show in Section~\ref{s.prflem} how the statement  extends to the
general case.

Throughout the paper   $A_i$ denotes the $i$--th row of a matrix $A$
and   $A^j$ its $j$--th column. Recall that a square integer matrix
is unimodular if it has determinant $\pm 1$. We also 
recall some standard facts on linear maps on abelian groups.
 Let $G$ be an abelian group of order $n$ and let $d$ be an integer coprime with $n$. Then the map $\phi_d:G\to G$ defined by multiplication, $g\mapsto d\cdot g$, is bijective and there is an integer $d'$ such that $\phi_d^{-1}(g)=d'\cdot g$. More generally, if $B$ is an  integer square matrix  of order $k$ with determinant $d$ coprime with $n$ then the linear map $\lambda_B:G^k\to G^k$, $g\mapsto Bg$, is also invertible and there is an integer matrix $B'$ such that $\lambda_B^{-1}(g)=B'g$.

\begin{definition}[Standard $n$--circular matrix]We say that a $(k\times m)$ integer matrix is 
{\it standard $n$--circular} if the
following properties hold:

\begin{itemize}
\item[{\rm (S1)}] $A=(I_k|B)$, where $I_k$ denotes the identity
matrix of order $k$.
\item[{\rm (S2)}] For each $j=1,\ldots ,m$,  the determinant formed by $k$ consecutive
columns in the circular order, $\{ A^{j+1}, A^{j+2},\ldots ,A^{j+k
}\}$ is coprime with $n$, where the superscripts are taken modulo $m$.
\end{itemize}
\end{definition}

We simply call matrices satisfying property S2 {\it $n$--circular}. 
Note that property S1 can always be imposed to an
$n$--circular  matrix by using elementary matrix
transformations  (with multiplication of rows only by  integers coprime to $n$). 

The next key Lemma  proves Theorem \ref{t.mainb} for
$n$--circular  matrices and abelain groups of order $n$ by constructing an hypergraph
associated to a given linear system. The approach is similar to the
one by Candela~\cite{can09} and by the authors
\cite{KralSerraVena09}.

\begin{lemma}\label{lem:red}
For each $\epsilon >0$ and positive integer $m$,  there is a $\delta>0$ depending only on $\epsilon$ and $m$ such that the following holds.

Let $A$ be a $k\times m$, $k\le m$, standard $n$--circular matrix and let $G$ be an abelian group of order $n$.
Let $X_1,\ldots ,X_n$ be a collection of subsets of $G$. 

If the number of solutions of the system $Ax=0$ with $x\in \prod_{i=1}^{m} X_i$ is at most
$\delta n^{m-k}$, then there are subsets $X_i'\subset X_i$ with
$|X_i'|< \epsilon n$ for all $i$ such that there is no solution of
the system $Ax=0$ with $x\in \prod_{i=1}^{m}\left(X_i\setminus
X_i'\right)$.

Moreover, if we have $X_j=G$, for  $j\in I$, where $I\subset
\{1,\ldots ,m\}$ has cardinality $|I|\le k$,
 then we can choose the sets $X_i'$ in such a way that
$X_j'=\emptyset$ for each $j\in I$. 
\end{lemma}

\begin{proof} We start by defining an integer $(m\times m)$ matrix
$C$ from which we will construct a pair of  colored hypergraphs $H$
and $K$. The purpose of this construction is to establish a
correspondence between solutions of the system $Ax=0$ with copies of
$H$ in $K$.

By property S2, the $j$--th column of $A$ can be written, for every
$j$, as an integer linear combination of the preceding $k$ columns
in the circular ordering:
\begin{displaymath}
    A^{j}=\sum_{i=j-k}^{j-1}C_{i,j} A^i,
\end{displaymath}
where the superscript $i$ is taken modulo $m$.

For $j=1,2,\ldots ,m$ we let  $C_{j,j}=-1$ and, if $i$ does not
belong to the circular interval $[j-k,j]$, then we set $C_{i,j}=0$ .
Thus,

\begin{equation}\label{eq:colj}
\sum_{i=1}^m C_{i,j}A^i=0,\; j=1,2,\ldots ,m.
\end{equation}

The integer $(m\times m)$ matrix $C=(C_{i,j})$ will be used to
define our hypergraph model for the given linear system.

Let $H$ be the following $(k+1)$-uniform colored  hypergraph. The vertex set of $H$ is
 $\{1,2,\ldots, m\}$. The edges of $H$ are the $m$
 $(k+1)$--subsets with consecutive elements in the circular ordering 
$$
\{ 1,\ldots,k+1 \}, \{2,\ldots,k+2\},\ldots, \{ m,1,\ldots,k\},
$$ 
(entries taken modulo $m$). The $i$-th edge $\{ i,i+1,\ldots,i+k\}$ is colored with color
$i$. Since $m\geq k+2$, $H$ contains $m$ different edges of mutually
different colors.

We next define the $(k+1)$-uniform colored  hypergraph $K$ as follows. Its vertex set is
$G\times [1,m]$. For each element $a_i\in X_i$, the $(k+1)$--subset
$$
\{(g_i,i),\ldots,(g_{i+k}, i+k)\}
$$ 
forms an edge labelled $a_i$ and colored with color $i$ if
\begin{equation}\label{eq:ai}
    a_i=\sum_{j=i}^{i+k}C_{i,j} g_j.
\end{equation}
Thus the edges of $K$ bear both, a color and a label. Observe that, for
each fixed $a_i\in X_i$, the system \eqref{eq:ai} has $n^{k}$
solutions in the $g_i$'s. Indeed, since $C_{i,i}= -1$, we can fix arbitrary values
$g_{i+1},\ldots ,g_{i+k}$ and get a value for $g_i$ satisfying the
equation. Therefore each element $a_i\in X_i$ gives rise to $n^k$
edges colored $i$ and labeled $a_i$. We also note that, by the construction of the matrix $C$, the equality \eqref{eq:ai} can be written as $a_i=\sum_{j=1}^m C_{i,j}g_j$.

 We next show that each solution to $Ax=0$
creates $n^{k}$ edge-disjoint copies of the hypergraph $H$ inside
$K$ and, also, that each copy of $H$ inside $K$ comes from a
solution of the system $Ax=0$.

\begin{claim} \label{claim1}
For any solution $x=(x_1,\ldots,x_{m})$ of the system
$Ax=0$ with $x_i\in X_i$, there are precisely $n^{k}$
edge--disjoint copies of the edge--colored hypergraph $H$ in the
hypergraph $K$ with edges labelled with $x_1,\ldots, x_m$.
\end{claim}

\begin{proof}
Fix a solution $x=(x_1,\ldots, x_{m})$ of $Ax=0$ with
$x_i\in X_i$, $1\le i\le m$.

Observe that, by property S2, $x$ is uniquely determined by any
of its subsequences $(x_i, x_{i+1},\ldots
, x_{i+m-k-1})$ of $m-k$  consecutive coordinates in the
circular ordering.

Recall that, by construction, each column $C^i$ of $C$ has zero 
entries in the rows $j\in [1,m]\setminus [i, i+1,\ldots ,i+k]$ (indexes modulo $m$ in $[1,m]$
here and in the sequel)  and its $i$--th entry is $-1$.

Therefore, for each choice of a vector $(g_{i+1},\ldots,\linebreak[1]g_{i+k})\in G^{k}$, 
there is a unique vector $ (g_{i+k+1}, \ldots ,g_{i-1},g_i)\in
G^{m-k}$ which satisfies the system $Cg=x$, where $x$ is the fixed solution and
 $g=(g_1,g_2,\ldots ,g_m)$. Indeed,  for each $t$, once
the values $(g_{i+1-t}, g_{i+2-t},\ldots, g_{i+k-t})$ have been
found, we can determine $g_{i-t}$ from the equation
\begin{equation}\label{eq:H}
x_{i-t}=\sum_{s=i-t}^{i+k-t} C_{i-t,s} g_{s},
\end{equation}
since $C_{i-t,i-t}= -1$. 

In this way, starting with the vector $
(g_{i+1}, \ldots ,g_{i+k-1},g_{i+k})\in G^{k}$ and the $m-k$ consecutive
elements $\{x_{i+k+1},\ldots,x_{i-1},x_{i}\}$ of $x$, we find a
unique $m$-dimensional vector $g=(g_1,\ldots,g_m)$ satisfying $Cg=x$. 

Moreover, if we let
$y =Cg\in G^m$, then $y$ satisfies $Ay =A(Cg)=(AC)g=0g=0$. Therefore
$y$ is a solution  of the system $Ax=0$ which shares $m-k$
consecutive values with the given solution $x$, hence
$y=x$. It follows that   the equations \eqref{eq:H} hold
for all $t$. Since    these are the  defining equations
\eqref{eq:ai} for the $k$--tuple $(g_i, i), \ldots ,(g_{i+k},i+k)$
to be an edge of $K$ colored $i$ and labeled $x_i$, we conclude that
each vector $(g_{i+1},\ldots,g_{i+k})\in G^{k}$ uniquely defines a
copy of $H$ in $K$.  Hence the solution $x$ induces $n^{k}$
copies of $H$ in $K$.

Let us show that these $n^k$ copies are edge disjoint. 
Recall that each entry $x_i\in X_i$ of $x$ gives rise to
$n^{k}$ edges labeled $x_i$ in the hypergraph $K$. On the other
hand, each of these edges belong to a unique copy of $H$ inside $K$
related to the solution $x$. Since this holds for each of the
edges and for each $x_i$, $1\leq i\leq m$, we conclude that the
$n^{k}$ copies of $H$ with edges labelled with
$x_1,\ldots,x_m$ are edge-disjoint.
\end{proof}

\begin{claim} \label{claim2} If $H'$ is a copy of $H$ in $K$, then
$x=(x_1,\ldots,x_{m})$ is a solution of the system, where $x_i$ is
the label of the edge colored by $i$ in $H'$.
\end{claim}

\begin{proof} The copy $H'$ has an edge of each color and is
supported over $m$ vertices. Indeed, since the edge colored $i$ contains a
vertex in $G\times \{ i\}$, then the copy $H'$ has one vertex on
each $G\times \{ i\}$, $1\leq i\leq m$. Hence the vertex set of $H'$
is of the form $\{ (g_1,1), (g_2,2), \ldots,(g_{m},m)\}$ for some
$g_1,\ldots ,g_{m}\in G$. If the edge $((g_i,i),\ldots
,(g_{i+k},i+k))$ colored $i$ in $H'$ has label $x_i$ then, by the
construction of $K$, we have $x_i=\sum_{s} C_{i,s}g_s$. Therefore,
it holds that $Cg=x$ where $g=(g_1, g_2, \ldots ,g_{m})$. Hence, as
all the columns in $C$ are in the kernel of $A$, we have
 $0=ACg=Ax$ and $x$ is a solution of the system.
\end{proof}

Claims \ref{claim1} and \ref{claim2} show that there is a correspondence
between the solutions of the system  $Ax=0$, with $x_i\in X_i$ for each $i$, and the copies of $H$
inside $K$. More precisely, each solution appears in the ordered labels of $n^k$ edge--disjoint copies of $H$ in $K$, and the labels of each copy of $H$ in $K$ form a solution. 


We now proceed with the proof of Lemma~\ref{lem:red}. Given
$\epsilon>0$ let $\delta>0$ be the value given by the Removal Lemma
of colored hypergraphs (Theorem~\ref{th:rmhc}) for the positive
integers $r=m, k'=k+1$ and $\epsilon'=\epsilon/m>0$.  

If the number of
solutions of the system $Ax=0$ is at most   $\delta n^{m-k}$, it
follows from Claims~\ref{claim1} and \ref{claim2}, that $K$ contains
at most $\delta n^{m}$ copies of $H$. By Theorem~\ref{th:rmhc}, there is a
set $E'$ of edges of $K$ with size   $\epsilon' n^{k+1}$ such that,
by deleting the edges in $E'$ from $K$, the resulting hypergraph is
$H$-free.

The subsets $X_i'\subset X_i$ of removed elements are constructed as
follows: if $E'$ contains at least $n^{k}/m$ edges colored with $i$
and labeled with $x_i$, we remove $x_i$ from $X_i$ (that is, $x_i\in
X_i'$.) In this way, the total number of elements removed from all
the sets $X_i$ together is at most $m\epsilon' n=\epsilon n$. Hence,
$|X_i'|\le \epsilon n$ as desired. Suppose that there is still a
solution $x=\left(x_1,x_2, \ldots,x_{m}\right)$ with $x_i\in
X_i\setminus X_i'$. Consider the $n^{k}$ edge--disjoint copies of
$H$ in $K$ corresponding to $x$. Since each of these $n^{k}$ copies
contains at least one edge from the set $E'$ and the copies are
edge--disjoint, $E'$ contains at least $n^{k}/m$ edges with the same
color $i$ and the same label $x_i$ for some $i$. However, such $x_i$
should have been removed from $X_i$, a contradiction.

It remains to show the last part of Lemma~\ref{lem:red}. Let $I$ be
a subset of $[1,m]$ with $|I|\leq k$, and suppose  that $X_j=G$ for
each $j\in I$. Let $H_0$ be the subgraph of $H$ formed by all the
edges in $H$ except the ones colored with $i\in I$. Note that $H$
contains a single copy of $H_0$. 

Since every vertex of $H$ belongs to
$(k+1)$ edges, the subgraph $H_0$ has no isolated vertices. It follows
that a copy $H_0'$ of $H_0$ in $K$ has precisely one vertex in $G\times
\{i\}$ for each $i=1,2,\ldots ,m$. By the construction of $K$, there
is at most one copy $H'$ of $H$ in $K$ containing $H_0'$, namely the
one whose labels are given by equation \eqref{eq:ai} given the
$g_i$'s. 

Since $X_j=G$ for each $j\in I$, then the label of each edge of $H'$ which is
missing in $H_0'$,  belongs to the
corresponding set $X_j$.  Hence such an edge is indeed present in $K$.
Hence, every copy of $H_0$ in $K$ can be uniquely extended to a copy
of $H$. We  conclude that $K$ contains as many copies of $H$ as of $H_0$. We can
apply Theorem~\ref{th:rmhc} to $H_0$ in the above argument to remove
all copies of $H_0$ by removing only elements from sets $X_i$ with
$i\in \{1,\dots ,m\}\setminus I$. This completes the proof.
\end{proof}

The condition $m\ge k+2$ in the hypothesis of Lemma \ref{lem:red}
has been used in the proof for the construction of the hypergraphs
associated to the linear system. However, this condition is not
restrictive for the proof of Theorem~\ref{t.mainb}; in the remaining cases
(when $m$ is $k$ or $k+1$),
we apply the following lemma:

\begin{lemma}\label{lem:k+2}
Let $A=(I_k|B)$ be a $(k\times m)$ integer matrix. If $m=\{k, k+1\}$
then the statement of  Theorem \ref{t.mainb} holds for $A$.
\end{lemma}

 \begin{proof}  For
$m=k$ the system has a unique solution and there is nothing to
prove. Suppose that $m=k+1$. Then, for each element $\alpha \in
X_{k+1}$ there is at most one solution to the system $Ax=b$ with
last coordinate $x_{k+1}=\alpha$. Let $X_{k+1}'$ be the set of
elements $\alpha\in X_{k+1}$ such that $x_{k+1}=\alpha$ is the last
coordinate of  some solution $x$. Since there are at most $\delta n$
solutions we have $|X_{k+1}'|\le \delta n$ and we are done by
removing the set $X_{k+1}'$. Thus the statement of Theorem
\ref{t.mainb} holds with $\delta =\epsilon$.
\end{proof}

\section{A reduction lemma}\label{s.prflem}

 In this section we prove some technical lemmas that will allow us
to derive Theorem~\ref{t.mainb}
from  Lemma~\ref{lem:red} via a series of transformations to the given linear
 system.

Recall that the adjugate of a square matrix  $L$, denoted by $\text{adj}(L)$,
is the matrix $C$ with $C_{i,j}=(-1)^{i+j} M_{j,i}(L)$, where $M_{j,i}(L)$
is the determinant of the matrix $L$ with the row $j$ and the column $i$ deleted.

Throughout the section $G$ denotes  an Abelian finite group of order
$n$. 

We start with some definitions which formalize our setting.

\begin{definition}[Restricted system]
	A {\rm restricted system} is a triple $\{ A,b, \X\}$ with
	\begin{itemize}
		\item $\X =X_1\times X_2\times \cdots \times X_m$, where $X_1,\ldots ,X_m$ are subsets of $G$. 
		\item $A$ is a $(k\times m)$ integer matrix such that its
$k$-th determinantal $d_k(A)$ satisfies $\gcd(d_k(A),|G|)=1$.
\item $b$ is an element of $G^k$, and we usually refer to it as the independent vector.
\end{itemize}
A { solution of the restricted system}
$\{ A,b, \X\}$ is a vector $x=(x_1,\dots ,x_m)\in G^m$ such that
$Ax=b$ and $x_i\in X_i$, $i=1,2,\ldots ,m$.
\end{definition} 

\begin{definition}[Extension of a restricted system]
	A restricted system $\{ A', b', \Y\}$ is an {\rm extension} of $\{ A,b, \X\}$ 	if the following  conditions hold:
	\begin{itemize}
	\item[{\rm E1:}] The dimensions $(k'\times m')$ of $A'$ and $(k\times m)$ of $A$ satisfy $k'\ge k$, $m'\ge m$, and $m'-k'=m-k$;
	\item[{\rm E2:}] There is a subset $I_0\subset [1,m']$ with cardinality $|I_0|=m$ such that
	$Y_i=G$ for each $i\in [1,m']\setminus I_0$; and
	\item[{\rm E3:}] There is a bijection $\sigma :I_0\rightarrow [1,m]$ and maps  $\phi_i
	:Y_i\rightarrow X_{\sigma (i)}$ such that the map 
	$$\phi: \Y
	\rightarrow \X$$
	defined as 
	\begin{equation}\label{def:phi}
	(\phi(y))_i=\phi_{\sigma^{-1}(i)}(y_{\sigma^{-1}(i)})
	\end{equation}
	 induces a
	bijection between the set of solutions of $\{A',b',\Y\}$ and the set
	of solutions of  $\{A,b,\X\}$.
	\end{itemize}
\end{definition}

Thus, an extension  $\{ A', b', \Y\}$  of $\{ A, b, \X\}$  has the
same number of solutions and one can define  a map $\phi$ with the following property. Denote by  
$$
\Y\setminus \Y'=\prod_{i=1}^{m'}Y_i\setminus Y_i'\; \mbox{ and } \X \setminus \phi (\Y')=\prod_{i=1}^{m}X_i\setminus \phi_{\sigma^ {-1}(i)}(Y_{\sigma^{-1}(i)}').
$$
Assume that $Y_j'=\emptyset$ when $j\not\in I_0$. Then, if the restricted system  $\{A', b',\Y\setminus \Y'\}$ has no solutions, then $\{ A, b, \X \setminus \phi (\Y')\}$   has no solutions either.

When $\{ A', b', \Y\}$ is an extension of $\{ A, b, \X\}$  with
$k=k'$,  any bijection for $\sigma$, and  the $\phi_i$'s are
bijective for each $i$, we say that the two systems are {\it equivalent}.

The purpose of this section is to  show that any restricted system
which fulfills the hypothesis of Theorem \ref{t.mainb} can be
extended to an homogeneous one with  a standard $n$--circular  matrix.
This will lead to a proof of Theorem \ref{t.mainb}   from Lemma
\ref{lem:red}.

The first easy step is to reduce the restricted system to an homogenous one.

\begin{lemma} If the restricted system $(A, b, \X)$ has a solution then it is equivalent to a restricted system $(A, 0,\X')$.
\end{lemma}

\begin{proof}  Choose a solution $y=(y_1,\ldots ,y_m)$ of $Ax=b$ and replace $\X$ by $\X'=(X_1-y_1)\times \cdots \times (X_m-y_m)$, so that a solution of $Ax=b$ satisfies $x\in \X$ if and only if $x-y\in \X'$, $x-y$ being a solution of the homogeneous linear system. 
\end{proof}

We next show that the matrix $A$ can be enlarged to an integer
square matrix $M$ of order $m$ such that $\det (M)=d_k(A)$. The
following Lemma uses the ideas of Zhan~\cite{Zhan} and Fang
\cite{Fang} to extend partial integral matrices to unimodular ones.
We include the proof of the simpler version we need for our
purposes.

\begin{lemma}[Matrix extension]\label{lem:ext-mat}
Let $A$ be a $k\times m$ integer matrix, $m\geq k$. Let $d_k(A)$
denote the greatest common divisor of the determinants of the
$\binom{m}{k}$ square $k\times k$ submatrices of $A$.

There is an $m\times m$ integer matrix $M$ such that

{\rm (i)} $M$ contains $A$ in its  $k$ first rows, and

{\rm (ii)} $\det(M)=d_k(A)$.
\end{lemma}

\begin{proof} Let $S=U^{-1}AV^{-1}$ be the Smith Normal Form of $A$, where $U$
and  $V$ are unimodular matrices. We have $S=\left(D|0\right)$,
where $D$ is a $k\times k$ diagonal integer matrix with
$|\det(D)|=|d_{k}(A)|$ and $0$ is an all--zero $k\times (m-k)$
matrix.

Recall that $U$ and $V$ are the row and column operations
respectively which transform $A$ into $S$. Observe that the row
operations do not modify the value of the determinant of any
$(k\times k)$ square submatrix of $A$. The column operations may
modify individual determinants but do not change the value of
$d_{k}(A)$.

Let $\overline{S}$ be the matrix:
\begin{displaymath}
\overline{S}=
\begin{pmatrix}
D & 0 \\
0 & I_{m-k}\\
\end{pmatrix},
\end{displaymath}
where $I_{m-k}$ denotes the identity matrix of order $m-k$.  We have
$\det(\overline{S})=\det(D)=d_k(A)$.

Then, if we let $\overline{V}=V$ and
\begin{displaymath}
\overline{U}=
\begin{pmatrix}
U & 0 \\
0 & I_{s-r}\\
\end{pmatrix},
\end{displaymath}
we obtain the matrix 
$$
M=\overline{U}\; \overline{S} \;
\overline{V}
$$ 
which clearly (i) contains $A$ as a submatrix in its
first $k$ rows, and (ii)
$\det(M)=\det(\overline{S})=d_k(A)$, since $\overline{U}$
and $\overline{V}$ are still unimodular.
\end{proof}

We say that the restricted system $\{ A,b, \X\}$ is {\it thin} if
the set of solutions is a subset of  $X_1\times \cdots \times
X_{j-1} \times \{\gamma_j\} \times X_{j+1} \times \cdots \times
X_m$, for some $j$ and $\gamma_j\in X_j$. Note that the statement of
Theorem~\ref{t.mainb} is obvious if the system is thin since it
suffices to delete the element $\gamma_j$ to remove all solutions.
Thus there is no loss of generality in assuming that our restricted
system is not thin.

\begin{lemma}\label{lem:ext0} The restricted system $\{A,0,\X\}$ is either thin or it
has an extension $\{ A',0,\Y\}$ such that
\begin{itemize}
\item[{\rm (i)}] $k'=m$ and $m'=2m-k$;
\item[{\rm (ii)}] the matrix $A'$ has the form $A'=(I_{k'}|B)$;
\item[{\rm (iii)}] $\gcd (B_i)=1$, where $B_i$ denotes the $i$--row of the
submatrix $B$;
\item[{\rm (iv)}] for every $k'<j\leq m'$, the restricting set $Y_j$ is the whole group $G$.
\end{itemize}
\end{lemma}

\begin{proof} By using Lemma~\ref{lem:ext-mat} we extend the matrix $A$ into
an $m\times m$ square matrix $$M=\left(
                                  \begin{array}{c}
                                    A \\
                                    E \\
                                  \end{array}
                                \right)
$$ with determinant $\det
(M)=d_k(A)$. We complete the square matrix $M$ to the $m\times
(2m-k)$   matrix
\begin{displaymath}
M'=
\begin{pmatrix}
A & 0 \\
E & I_{m-k}\\
\end{pmatrix}=(M|B').
\end{displaymath}

We now consider the restricted  system $\{ M',0,\X'\}$ where
$$X'_i=\left\{
              \begin{array}{ll}
                X_i, & 1\le i\le m; \\
                G, & m+1\le i\le 2m-k.
              \end{array}
            \right.
$$
By letting $I_0=[1,m]$ and $\sigma$ and $\phi_i$ be the identity
maps we see that the function $\phi$ as defined in \eqref{def:phi} induces 
a bijection between the solutions of
$\{ M',0,\X'\}$ and the solutions of $\{ A,0,\X\}$. 
 Therefore $\{ M',0,\X'\}$ is an extension of the
original system.

Let $U={\rm adj}(M)$ denote the adjugate of $M$. Since $d=d_k(A)$ is
relatively prime with $n$,  the matrix $U$ is invertible and we get an equivalent restricted system
$\{ M'',0,\X'\}$ by setting
$$
M''=(U  M|UB')=(d\cdot I_m|UB').
$$
Let $d'$ be an integer such that $g=dh$ if and only if $d'g=h$ for each $h\in G$. 
By replacing each $X_i'$, for $i\in [1,m]$, by
$\X''_i=d'\cdot X_i'$ and $X''_i=X_i'$, for $i\in
[m+1,2m-k]$, we get a an equivalent system of the form $\{
(I_{m}|B''), 0,\X''\}$ where $B''=UB'$.  At this point we have 
an equivalent system which satisfies the conditions (i), (ii) and (iv) of the Lemma.

We observe that, if $B_j^{''}=0$ for some row $j$ of $B''$, then the $j$-th
equation implies $x_j=0$. Thus, the solution set of $\{ (I_{m}|B''),
0,\X''\}$ is inside $X_1^{''}\times \cdots \times X_{j-1}^{''}
\times \{0\} \times X_{j+1}^{''} \times \cdots \times X_{m'}^{''}$,
which implies that the solution set for the original system is
inside $X_1\times \cdots \times X_{j'-1} \times \{\gamma_{j'}\} \times
X_{j'+1} \times \cdots \times X_{m}$, for some $j'$ and some $\gamma_{j'}\in
X_{j'}$. Thus, if $B_j^{''}=0$, then the system is thin. Therefore
we can assume that all the rows in $B^{''}$ are non--zero.

Suppose that $\gcd(B_i'')=s>1$, where $B_i''$ denotes the $i$--th
row of $B''$. Then the $i$--th coordinate $y_i$, $i\in[1,m]$, of a
solution of $(I_m|B'')y=0$ belongs to the subgroup $s\cdot G$ of
$G$. Thus we may assume that $X_i''\subset  s\cdot G$. Let
$Y_i=s^{-1}(X''_i)$, where now $s^{-1}$ denotes the preimage of the
canonical projection $s:G\rightarrow s\cdot G$ defined by $s(g)=sg$,
and divide the entries of the $i$--row $B''_i$ by $s$. In this way
we obtain an extension of $\{ (I_{m}|B''), 0,\X''\}$ where the map
$\phi_i:Y_i\rightarrow X_i''$, $i\in[1,m]$, is the multiplication by
$s$. Notice that, even though $\phi_i$ is not a bijection, the map $\phi$ as defined in \eqref{def:phi} does induce a bijection between the set of solutions of $\{ (I_{m}|B''), 0,\X''\}$ and the ones of the new system, since  different solutions $y, y'$ of the new system can be distinguished by the value of $(0,-B''_i/s)\cdot y$. Moreover, if $X_i''=G$ then $Y_i=G$ as well.

By repeating the same procedure with each row of $B''$ we
eventually obtain an extension $\{ A',0,\Y\}$ satisfying the
conditions (i)--(iv) of the Lemma.  
This completes
the proof.
\end{proof}

For our last step we will use the following technical Lemma.

\begin{lemma}\label{lem:ext1} Let $n$ be a positive integer and let $M$ be an $r\times r$ integer matrix with determinant $d$ coprime with $n$. There are integer matrices $S$ and $T$ such that 
	\begin{displaymath}
		\overline{M}=\left(
		\begin{array}{c}
			I_r  \\
			S \\
			M \\
			T\\
			I_r\\
			\end{array}\right)
	\end{displaymath}
is a $s\times r$, $s=r(2r+1)$, integer matrix with the property that each $r\times r$ submatrix of $\overline{M}$ consisting of $r$ consecutive rows has a determinant coprime with $n$.
\end{lemma}

\begin{proof}  For $t\ge r$ let us say that an $t\times r$ integer matrix $A$ is $n$--good if every submatrix of $A$ formed by $r$ consecutive rows has determinant coprime with $n$.

  We first show that there is a matrix $T$ such that
\begin{displaymath}
		N=\left(
		\begin{array}{c}
			M \\
			T\\
			I_r\\
			\end{array}\right)
	\end{displaymath}
is a $(r(r+1))\times r$ integer matrix which is $n$--good. 

We proceed by induction on $r$. For $r=1$ we can write $N=\left(\begin{array}{c} d\\ 1\end{array}\right)$. Let $r>1$. We construct  the matrix  $N$ by adding rows one by one to the bottom of $M$. We first observe that  the matrix formed by the rows 
\begin{displaymath}
		\left(
		\begin{array}{c}
			M_2 \\
			M_3\\
			\vdots\\
			M_r\\
			\sum_{i=1}^r \lambda_i M_i
			\end{array}\right)
	\end{displaymath}   
has determinant $d\lambda_1$. Let 
\begin{equation}\label{eq:d'}
\lambda_1M_{1,1}+\lambda_{2}M_{2,1}+\cdots +\lambda_r M_{r,1}=d'
\end{equation}
be an integer linear combination of  the entries in the first column of  $M$, where $d'$ is its greatest common divisor. Note that $d'$ divides $d$, hence it is also coprime with $n$. We can choose $\lambda_1$ to run on an arithmetic  progression $a+b\z$ with $\gcd(a,b)=1$, by keeping the identity \eqref{eq:d'} with appropriate values of $\lambda_2,\ldots ,\lambda_r$. Thus, by Dirichlet theorem, we may choose $\lambda_1$ to be some prime larger than $n$. We define the first row of $T$ to be 
$$
T_1= \sum_{i=1}^r \lambda_i M_i
$$
for the above choice of $\lambda_1, \lambda_2,\cdots ,\lambda_r$. In this way, the matrix
$$
\left( \begin{array}{c} M\\ T_1\end{array}\right)
$$
is $n$--good. Moreover, $T_{11}=\gcd (M^1)=d'$.

We next proceed to add the next $r-1$ rows. For   $i=2,\ldots,r$, we define  $T_i=M_i-(M_{i,1}/d') T_1$. By the remark at the beginning of the proof, the matrix
$$
\left( \begin{array}{c} M\\ T_1\\T_2 \\ \vdots \\ T_r \end{array}\right)= \left( \begin{array}{c} M\\ T' \end{array}\right)
$$
is $n$--good,  where
 $$
 T'=\left( \begin{array}{cc}  d' & \ast\\0&M'\end{array}\right),
 $$
for some integer square matrix $M'$ of order $r-1$ which has determinant coprime with $n$.

By induction hypothesis there is $T''$ such that the $r(r-1)\times r$ integer matrix
\begin{displaymath}
		N'=\left(
		\begin{array}{c}
			M' \\
			T''\\
			I_{r-1}\\
			\end{array}\right)
	\end{displaymath}
is $n$--good. Add to $N'$ a first column of zeros and insert in the resulting matrix  the row $(1,0,\ldots ,0)$ of length $r$ between the  positions $j(r-1)$ and $j(r-1)+1$ for $j=1,\ldots, r-1$. Moreover, insert the row $T_1=(d' \; \ast)$ as the first row of $N'$. The resulting matrix $N''$ has $r\left(r-1\right)+ r=r^2$ rows and, by construction,
\begin{displaymath}
		N=\left(
		\begin{array}{c}
			M \\
			N''
			\end{array}\right)=\left(\begin{array}{c}
				M \\
				T\\
				I_r
				\end{array}\right)
	\end{displaymath}
is $n$--good, has $r(r+1)$ rows and $r$ columns and has the desired form for some matrix $T$.

By the same argument adding rows to the top of $M$ we see that there is also a matrix $S$ which, combined with $T$, gives the result.	
\end{proof}

Our final step is to show that, if the restricted system $\{
A,0,\X\}$, where $A$ satisfies the conclusions of Lemma
\ref{lem:ext0}, is non--thin, then it admits an extension with a standard
$n$--circular matrix.

\begin{lemma}\label{lem:ext2} Let $\{ A,0,\X\}$ be a  non--thin
 restricted system where $A=(I_k|B)$ and $\gcd(B_i)=1$ for every row $i$.
 There is an extension
$\{ A',0,\X'\}$ with the following properties.
\begin{itemize}
\item[{\rm (i)}] $A'$ is a standard $n$--circular matrix;
\item[{\rm (ii)}] the dimensions of $A'$ depend only on the dimensions of $A$; and
\item[{\rm (iii)}] up to a reordering
of the subscripts, $\X'=\X\times \prod_{j=m+1}^{k'+m-k} G $.
\end{itemize}
\end{lemma}

\begin{proof} We apply Lemma \ref{lem:ext1} to  the matrix $B$ in the following
manner.  As each row $B_i$ of the submatrix $B$ is such that
$\gcd(B_i)=1$, we can apply Lemma~\ref{lem:ext-mat} to the row $B_i$, by
to obtain a $(m-k)\times (m-k)$ square matrix $\overline{B_i}$ with
determinant $ 1$. Thus, by applying Lemma \ref{lem:ext1} to each
of the resulting matrices $\overline{B_1},\ldots ,\overline{B_k}$ we
may construct the following $k'\times (m-k)$ rectangular  matrix:
\begin{displaymath}
B'=
\begin{pmatrix}
I_{m-k} \\
S_1 \\
\overline{B_1} \\
T_1 \\
I_{m-k} \\
S_2 \\
\overline{B_2}\\
T_2 \\
I_{m-k} \\
\cdots \\
I_{m-k} \\
S_k \\
\overline{B_k} \\
T_k \\
I_{m-k}
\end{pmatrix},
\end{displaymath}
for some   $k' $ depending on the dimensions of $B$. Let
$$ 
A'=(I_{k'}|B').
$$

Observe that every set of $k'$ consecutive columns in the circular
order in $A'$ form a  matrix with determinant coprime with $n$. Following our
terminology, $A'$ is standard $n$--circular. To check this, let $M(i)$
be the square submatrix  formed by $k'$ consecutive columns of $A'$
in the circular order starting with the $i$--th column.

Since the matrix $A'$ has the form
$$
A'=\left( I_{k'}\left| \begin{array}{c}
                         I_{m-k} \\
                         X \\
I_{m-k}
                   \end{array}\right.\right)$$
for some matrix $X$, then each matrix $M(i)$ for   $i=1,\ldots ,m-k$
is a circular permutation of a lower triangular matrix with all ones
in the diagonal. Moreover, if $i=m'-(m-k)+1,\ldots,m'$ then $M(i)$ is an upper triangular matrix with all ones in the diagonal. Hence $M(i)$  is unimodular for these values of
$i$.

For the remaining values of $i$, $\det M(i)$ equals, up to a sign,
the determinant of a submatrix of $B'$ formed by $m-k$ consecutive
rows. More precisely, $\det\left[
M((m-k)+t)\right]$ equals, up to a sign, the determinant of the matrix
formed by the rows $B'_{t+1}, B'_{t+2}, \ldots ,B'_{t+(m-k)}$. Since $B'$ is $n$-good, then $\gcd\left(\det M(i),n\right)=1$.

In order to complete the proof of the Lemma we must construct the
family $\X'$ of $m'=k'+m-k$ sets. Let $I_0^1\subset [1,k']$ be the
set of subscripts $i$ for which the $i$--row of $B'$ corresponds to a
row $\sigma (i)$ of the original matrix $B$ and let
$I_0^2=[k'+1,m']$. Let $I_0=I_0^1\cup I_0^2\subset [1,m']$. By setting
$$
X'_i=\left\{\begin{array}{ll}X_{\sigma (i)} & i\in I_0^1\\X_{i-m'+m}& i\in
I_0^2\\ G &\mbox{ otherwise}\end{array}\right.
$$
we get an extension $(A',0,\X')$ of
the given restricted system with
\begin{displaymath}
    \phi: \prod_{i=1}^k X'_{\sigma^{-1}(i)}\times \prod_{i=k+1}^m X'_{i+m'-m}
\rightarrow \prod_{i=1}^k X_i\times \prod_{i=k+1}^m X_i
\end{displaymath}
the identity map. This completes the proof.
\end{proof}

Observe that Lemma~\ref{lem:ext0} and Lemma~\ref{lem:ext2} can be
concatenated to obtain a single, coherent, extension. The variables
added in Lemma~\ref{lem:ext0}, that run over the whole group $G$,
will also be moving over $G$ after the second extension provided by
Lemma~\ref{lem:ext2}. We summarize the results of this section in
the following Proposition.

\begin{proposition}\label{prop:ext} Let $G$ be an Abelian group of order $n$.
Let $\{ A,b,\X\}$, where $A$ is an integer $(k\times m)$ matrix, be
a non--thin restricted system with $\gcd (d_k(A),n)=1$. 

There is an
extension $\{ A',b',\X'\}$ of $\{ A,b,\X\}$ such
that $A'$ is of the form $A'=(I_{k'}|B)$, $b'=0$ where $A'$ is a standard $n$--circular
matrix whose dimesnions depend only on the dimensions of $A$.
\end{proposition}

\section{Proof of Theorem \ref{t.mainb}}\label{sec:proof}

We complete here the proof of Theorem \ref{t.mainb}. We  assume that
the system is not thin, otherwise, the result holds by deleting just
one element of one set.

By Lemma \ref{lem:k+2} we may assume that $m'-k'\ge 2$. Let
$\epsilon>0$ and an integer $(k\times m)$ matrix $A$ be given. Let
$G$ be an  Abelian group of order $n$ coprime with $d_k(A)$, and let
$\{ A, b,\X\}$ be a restricted system in $G$. It follows from
Proposition \ref{prop:ext} that there is an extension $\{
A',0,\X'\}$ of $\{ A, b,\X\}$ such that $A'$ is a standard $n$--circular
  matrix of dimension $(k'\times m')$ with $m'-k'=m-k$ and
$k'=k'(m,k)$. Moreover there is a subset $I_0\subset [1,m']$ with
cardinality $m$, a bijection $\sigma :I_0\rightarrow [1,m]$ and maps
 $\phi_i :X_i'\rightarrow X_{\sigma (i)}$, $1\le i\le m$ such that
the map  $\phi: \X' \rightarrow \X$ with
$(\phi(x'))_i=\phi_{\sigma^{-1}(i)}(x'_{\sigma^{-1}(i)})$ induces a
bijection between the set of solutions of  $\{A',0,\X'\}$ and the
set of solutions of  $\{A,b,\X\}$. In addition, $I=[1,m']\setminus
I_0$ has cardinality less than $k'$ and $X_i'=G$ for each $i\in I$.

We apply Lemma~\ref{lem:red} to the extension $\{A',0,\X'\}$ to
obtain a set $\bar{\X}'$ with $|\bar{X}'_i|<\epsilon n$ for all
$i\in[1,m']$ such that $\{A',0,\X'\setminus \bar{\X}'\}$ has no
solution. We use the last part of Lemma~\ref{lem:red} to ensure that
$\bar{\X}'$ can be chosen in such a way that $\bar{X}'_i=\emptyset$
for each $i\in I=[1,m']\setminus I_0$. This shows that $\{ A, b,\X
\setminus \phi(\bar{\X}')\}$ is solution free and
$|(\phi(\bar{\X}'))_i|<\epsilon n$ for $i\in[1,m]$. This completes
the proof of Theorem~\ref{t.mainb}.

\section{On the condition on the determinantal} \label{s.dtmntl_cond}

As it has been pointed out to the authors in several occasions, the main result in this paper would be neater if the condition regarding the coprimality between the $k$--determinantal $d_k(A)$ of the matrix and order of the group $n$   could be removed from the hypothesis of the statement.

For a given group $G$, the condition ensures that the system $Ax=b$ has the appropriate number $n^{m-k}$ of solutions, and in that respect it is only natural that such a condition is placed in the statement. On the other hand, if $\gcd (n,d_k(A))=d>1$, this means that, in an equivalent system, one of the equations of the system is simply multiplied by some integer different than one, which is a somewhat unnatural situation. 

The proper statement in the general case should say that, if the linear system $Ax=b$ has less than $\delta S(A,G)$ solutions with entries in sets $X_1,\ldots ,X_n$, then it can be made solution--free by removing at most $\epsilon n$ elements in each set, where $S(A,G)$ denotes the total number of solutions of the system (which is larger than $n^{m-k}$). 

As it happens, such a statement does hold, but to the cost of making  $\delta$ depend on the particular entries of the matrix $A$ and not only on $\epsilon$ and the dimensions of $A$, thus modifying the nature of the statement of Theorem \ref{t.mainb}. On the other hand, proving such a result requires the development of a new statement of the removal lemma which has its own technical difficulties. Let us try to explain the reason for this.  

It can be shown that, if $d=\gcd (d_k(A),n)>1$ and the system $Ax=b$ does have solutions, then there are integers $\overline{d}_1|\ldots|\overline{d}_k$ with $\prod_{i=1}^k \overline{d}_i=d$, a matrix $A'$ and vectors $b_1,\ldots ,b_t$, where $t=\prod_{i=1}^k n/|\overline{d}_i\cdot G|$ such that the  set of solutions of $Ax=b$ is the union of the sets of solutions of the $t$ linear systems
$$
A'x=b_1, A'x=b_2,\ldots , A' x=b_t,
$$
and $\gcd(d_k(A'), n)=1$.  This decomposition can be combined with Theorem~\ref{t.mainb} to obtain an analogous statement without the condition $\gcd (d_k (A),|G|)=1$ for the family of cyclic groups $\z_n$, for example. 

However the strategy of  simply combining Theorem \ref{t.mainb} with the above decomposition  is far from sufficient to solve the problem for the general class of abelian groups. 
Combining all the sets solutions in a suitable form requires a new formulation of the removal lemma for product structures which, having an interest in its own, involves technical difficulties which are detailed in a forthcoming paper of one of the authors \cite{Vena12}. The latter   builds on the construction presented in this paper, which has indeed an interest in its own as being the natural generalization of  the version for finite fields, in which the  $\delta$ does not depend on the actual entries of the matrix $A$ but only on its dimensions.

\section*{Acknowledgements}

We would like to thank Andrew Thomason for his remarks on this paper
and an anonymous referee for useful comments and observations.

\bibliographystyle{abbrv}
\bibliography{Bib-mat-1}

\end{document}